\documentclass[11pt,reqno]{amsproc}

\title[]{Local and global strong solutions for SQG in bounded domains}

\author{Peter Constantin}
\address{Department of Mathematics, Princeton University, Princeton, NJ 08544}
\email{const@math.princeton.edu}

\author{Huy Quang Nguyen}
\address{Program in Applied and Computational Mathematics, Princeton University, Princeton, NJ 08544}
\email{qn@math.princeton.edu}

\usepackage[margin=1in]{geometry}
\usepackage{amsmath, amsthm, amssymb}
\usepackage{times}
\usepackage{color}
\usepackage{mathrsfs} 
\usepackage{hyperref}
\usepackage{tikz}
\usetikzlibrary{arrows}
\usepackage{verbatim}

\newcommand{\bq}{\begin{equation}}
\newcommand{\eq}{\end{equation}}
\newcommand{\bqa}{\begin{eqnarray*}}
\newcommand{\eqa}{\end{eqnarray*}}
\newcommand{\Rr}{\mathbb{R}}
\newcommand{\Zz}{\mathbb{Z}}

\newcommand{\Tt}{\mathbb{T}}

\newcommand{\la}{\label}

\newcommand{\na}{\nabla}
\newcommand{\be}{\begin{equation}}
\newcommand{\ee}{\end{equation}}
\newcommand{\ba}{\begin{array}{l}}
\newcommand{\ea}{\end{array}}
\theoremstyle{plain}
\newtheorem{theo}{Theorem}[section]
\newtheorem{prop}[theo]{Proposition}
\newtheorem{lemm}[theo]{Lemma}

\theoremstyle{definition}
\newtheorem{rema}[theo]{Remark}

\DeclareMathOperator{\cnx}{div}
\DeclareMathOperator{\cn}{div}

\DeclareMathOperator{\dist}{dist}

\DeclareMathOperator{\supp}{supp}

\DeclareSymbolFont{pletters}{OT1}{cmr}{m}{sl}
\DeclareMathSymbol{s}{\mathalpha}{pletters}{`s}

\def\tt{\theta}

\def\eps{\varepsilon}

\def\na{\nabla}
\def\la{\left\lvert}

\def\le{\leq}

\def\L#1{\langle #1 \rangle}

\def\mez{\frac{1}{2}}

\def\k{\kappa}

\def\ra{\right\rvert}

\def\tdm{\frac{3}{2}}

\def\P{\mathbb P}

\def\L{\Lambda}
\def\cF{\mathcal{F}}

\numberwithin{equation}{section}

\date{today}
\begin{document}
\begin{abstract}
We prove local well-posedness for the inviscid surface quasigeostrophic (SQG) equation in bounded domains of $\Rr^2$. When fractional Dirichlet Laplacian dissipation is added, global existence of strong solutions is obtained for small data for critical and supercritical cases. Global existence of strong solutions with arbitrary data is obtained in the subcritical cases.
\end{abstract}


\noindent

\begin{center}{\em To Edriss Titi, with friendship, respect and admiration}\end{center}

\maketitle
\section{Introduction}
Let $\Omega\subset \Rr^2$ be an open bounded set with smooth boundary. The surface quasigeostrophic (SQG) equation  in $\Omega$ is the equation
\bq\label{SQG}
\partial_t\tt  +u\cdot \nabla\tt +\k \L^{2\alpha}\tt=0,\quad\alpha\in(0,1),~\k\ge 0,
\eq
where 
\[
\L:=\sqrt{-\Delta}.
\]
The Laplacian $-\Delta$ above has homogeneous Dirichlet boundary conditions,
and the equation is an active scalar equation: the scalar $\tt = \tt(x,t)$ determines  $u = u(x,t)$ for  $(x, t)\in \Omega\times [0, \infty)$ by
\bq\label{u:defi}
u = R_D^\perp \tt :=\nabla^\perp \L^{-1}\tt.
\eq
The nonnegative number $\k$ distinguishes between the dissipative SQG equation \eqref{SQG}, when $\k>0$, and the inviscid SQG equation when $\k=0$.

The domain of the Laplacian $-\Delta$ with homogeneous Dirichlet boundary conditions is  
\[
D(-\Delta)=H^2(\Omega)\cap H^1_0(\Omega),
\]
and the fractional Laplacian $\L^s$, $s\ge 0$ is defined using eigenfunction expansions. The domain of definition of the fractional Laplacian, $D(\L^s)$ is endowed with a natural norm $\Vert\cdot\Vert_{s, D}$ and is a Hilbert space (see section \ref{pre} below for details). In particular, the norm of $D(\L^2)= D(-\Delta)$ is equivalent to the $H^2(\Omega)$ norm.

The main results of this paper concerning the dissipative SQG equation are the local well-posedness for the whole range of $\alpha\in (0, 1)$ for arbitrary data in $D(\L^2)$ and the existence of unque global solutions for small data in $D(\L^2)$.
\begin{theo}\label{main1}
Let $\alpha\in (0, 1)$ and $\k>0$. Let $\tt_0\in D(\L^2)$ be an initial datum.

1. There exists a constant $M$ depending only on $\alpha$, such that, on
 the time interval $[0, T]$, with
\[
T=\frac{\k}{M\Vert \tt_0\Vert^2_{2, D}},
\]
\eqref{SQG}  has a unique solution in 
\[
\tt\in L^\infty\big([0, T]; D(\L^2)\big)\cap L^2\big([0, T]; D(\L^{2+\alpha})\big).
\]

2. There exists a positive constant $C$ depending only on $\alpha$ such that the following holds: if 
\[
\Vert\tt_0\Vert_{2, D}< \frac{\k}{C}
\]
then there exists a unique global-in-time solution
\[
\tt\in L^\infty\big([0, \infty); D(\L^2)\big)\cap L^2_{\text{loc}}\big([0, \infty); D(\L^{2+\alpha})\big)
\]
of \eqref{SQG}. Moreover, the $D(\L^2)$ norm of $\tt$ is bounded by its initial value:
\[
\Vert \tt(t, \cdot)\Vert_{2, D}\le \Vert \tt_0\Vert_{2, D}\quad\text{a.e.}~t\ge 0.
\]
\end{theo}
The subcritical SQG equation \eqref{SQG} with $\alpha\in (\mez, 1)$ is globally well-posed, as in the case without boundaries:  
\begin{theo}\label{main3}
Let $\alpha\in (\mez, 1)$, $\k>0$, and $T>0$. Let $\tt_0\in D(\L^2)$ be an initial datum.  There exists a unique solution 
\bq\label{reg:sub:theo}
\tt\in L^\infty\big([0, T]; D(\L^2)\big)\cap L^2\big([0, T]; D(\L^{2+\alpha})\big)
\eq
of \eqref{SQG}. 
\end{theo}
The result of this paper concerning the inviscid SQG equation is the local well-posedness in a class of classical solutions.
\begin{theo}\label{wp:inviscid}
Let $p\in (2, \infty)$. For every $\tt_0\in H^1_0(\Omega)\cap W^{2, p}(\Omega)$, there exist $T=T(\Vert \tt_0\Vert_{H^1_0\cap W^{2, p}}, p)>0$ and unique solution 
\[
\tt\in L^\infty([0, T]; H^1_0(\Omega)\cap W^{2, p}(\Omega))
\]
 to \eqref{SQG} with $\k=0$.
\end{theo}
The surface quasigeostrophic equation of geophysical significance (\cite{held}) serves as a two-dimensional model for the three-dimensional Euler equations due to many  mathematical and physical analogies between them (\cite{cmt}). There is a vast literature devoted to local and global well-posedness issues for SQG in $\Rr^2$ and $\Tt^2$. It is known that $L^2$ global weak solutions exist for arbitrary data  (\cite{Res}). The subcritical dissipative case is well-understood (\cite{Res, ConWu, CorCor}) and global solutions with small initial data in the critical space for the critical SQG were obtained in \cite{ccw}. Global regularity for the critical dissipative case is subtle and was first obtained independently in \cite{CaVa, Kis}. There are several later proofs  of this result \cite{Kis2, ConVic}.  The global regularity for the supercritical dissipative and inviscid SQG are outstanding open problems. 

The study of SQG in bounded domains with smooth boundaries was initiated in \cite{ConIgn, ConIgn2} where  $L^2$ global weak solutions were obtained and global Lipschitz a priori interior estimates were obtained for critical SQG. $L^2$ global weak solutions for the inviscid SQG were obtained in \cite{ConNgu}, and generalized in \cite{Ng} for SQG-type equations with more singular constitutive laws, $u=\nabla^\perp \L^{-\beta}\tt$ with $\beta\in (0, 1)$. As in the cases without boundary, uniqueness of weak solutions is not known. The presence of boundaries makes the well-posedness issues become more delicate. The main source of difficulties is the lack of translation invariance of the fractional Laplacian in bounded domains. This manifests itself in particular in the commutator estimates for the fractional Laplacian. In order to appreciate these difficulties, let us consider the local well-posedness in Sobolev spaces for the inviscid SQG.  For the flow to be well-defined it is good for the velocity $u$  to be Lipschitz continuous, and so natural Sobolev spaces for local well-posedness (in two dimensions) are $H^s$ with $s>2$ (because $u$ is obtained from $\tt$ through Riesz transforms). The main tools for proving local well-posedness in the whole space  (\cite{cmt, CorCor}, see also \cite{Ju1}) are the well-known Kato-Ponce commutator estimate (\cite{KP})
\bq\label{KP}
\Vert [\L^s, u]\cdot \nabla \tt]\Vert_{L^2(\Rr^2)}\le C\Vert \nabla u\Vert_{L^\infty(\Rr^2)}\Vert \nabla\tt\Vert_{H^{s-1}(\Rr^2)} +C\Vert  u\Vert_{H^s(\Rr^2)}\Vert \nabla\tt\Vert_{L^\infty(\Rr^2)} \le C\Vert u\Vert_{H^s(\Rr^2)}\Vert \tt\Vert_{H^s(\Rr^2)}
\eq
with $s>2$, where  $\mathscr{F}(\L^sf)(\xi)=|\xi|^s(\mathscr{F}f)(\xi)$, with $\mathscr{F}$ denoting the Fourier transform. Additionally, it is useful that withe Riesz transforms are continuous in Sobolev spaces
\bq\label{cont:R}
\Vert R\tt\Vert_{H^r(\Rr^2)}\le C\Vert \tt\Vert_{H^r(\Rr^2)}\quad\forall r\ge 0.
\eq
The bound \eqref{cont:R} follows directly from the Plancheral theorem. In bounded domains the estimate \eqref{KP} fails because the fractional Laplacian does not commute with differentiation, and the existing sharp estimate  \cite{ConIgn} is too expensive. In order to do regularity calculations the commutator between $\L^s$ and $\nabla$ needs to be considered. This has a singular behavior at the boundary \cite{ConIgn2}, \cite{ConNgu} (which is sharp in half-space):
\[
\la [\L^s, \nabla]f(x)\ra\le \frac{C}{d(x)^{s+1+\frac{d}{p}}}\Vert f\Vert_{L^p(\Omega)}
\]
with $\Omega\subset \Rr^d$, $p\in [1, \infty]$, and $d(x)=\dist(x, \partial\Omega)$. In order to overcome this and to obtain local well-posedness in the inviscid case the idea is to take  even indices $s$, $s=2m$, because then $\L^{2m}$ commutes with $\nabla$ on  the domain $D(\L^{2m})$ of $\L^{2m}$.  This in turn however requires that the nonlinearity $u\cdot\nabla\tt$ to belong to $D(\L^{2m})$, provided $\tt\in D(\L^{2m})$. Unfortunately, this is not true in general. It is true for $m=1$ because $u\cdot\nabla\tt$ vanishes on the boundary. This is due to the following structure:  $u=\nabla^\perp\psi$ is tangent to the boundary because $\psi\vert_{\partial\Omega}=0$, and $\nabla\tt$ is normal to the boundary, because $\tt\vert_{\partial\Omega}=0$. Taking derivatives of $u\cdot \nabla\tt$ unfortunately breaks down this structure. Forced to work with $m=1$, we face another obstacle: $u\in D(\L^2)$ is not Lipschitz continuous. Therefore in Theorem \ref{wp:inviscid} we prove local well-posedness in $H^1_0(\Omega)\cap W^{2, p}(\Omega)$ with $p>2$, hence ensuring that $u$ is Lipschitz. The added difficulty now is that continuity of the Riezs transform from $W^{2, p}(\Omega)$ to $W^{2, p}(\Omega)$ is not available. The proof then consists of three bootstraps: Galerkin approximations to obtain the $H^2$ regularity, a transport estimate to obtain the $W^{2, q}(\Omega)$ regularity for any $q\in (2, p)$, and finally another transport estimate to gain the full $W^{2,p}(\Omega)$ regularity.

The paper is organized as follows. In section \ref{pre} we present the functional setup for the fractional Laplacian in domains using eigenfunction expansions. Theorems \ref{main1}, \ref{main3}, \ref{wp:inviscid} are proved in sections \ref{section:main1}, \ref{section:main3}, \ref{section:inviscid}, respectively. Appendices 1 and 2 are devoted to $L^p$ bounds and local well-psoedness for the linear advection-diffusion equations with fractional dissipation.
%%%%%%%%%%
\section{Preliminaries}\label{pre}
Let $\Omega$ be an open bounded set of $\Rr^d$, $d\ge 2$, with smooth boundary. The Laplacian $-\Delta$ is defined on $
D(-\Delta)=H^2(\Omega)\cap H^1_0(\Omega)$. Let $\{w_j\}_{j=1}^\infty$ be an orthonormal basis of $L^2(\Omega)$ comprised of $L^2-$normalized eigenfunctions $w_j$ of $-\Delta$, {\it {\it i.e.}}
\[
-\Delta w_j=\lambda_jw_j, \quad \int_{\Omega}w_j^2dx = 1,
\]
with $0<\lambda_1<\lambda_2\le...\le\lambda_j\to \infty$.\\
The fractional Laplacian is defined using eigenfunction expansions,
\[
\Lambda^{\alpha}f\equiv (-\Delta)^{\frac{\alpha}{2}} f:=\sum_{j=1}^\infty\lambda_j^{\frac{\alpha}{2}} f_j w_j\quad\text{with}~f=\sum_{j=1}^\infty f_jw_j,\quad f_j=\int_{\Omega} fw_jdx
\]
for $\alpha\ge 0$ and 
\[
f\in D(\Lambda^{\alpha}):=\{f\in L^2(\Omega): \big(\lambda_j^{\frac{\alpha}{2}} f_j\big)\in \ell^2(\mathbb N)\}.
\]
 The norm of $f$ in $D(\Lambda^{\alpha })$ is defined by
\[
\Vert f\Vert_{\alpha, D}:=\Vert\L^\alpha f\Vert_{L^2(\Omega)}=\big(\sum_{j=1}^\infty\lambda_j^\alpha f_j^2\big)^\mez.
\]
%For any $\alpha\ge 0$, let $D(\L^{-\alpha})$ denote the dual space of $D(\L^\alpha)$.\\
It is also well known that $D(\Lambda)$ and $H^1_0(\Omega)$ are isometric, where $H^1_0(\Omega)$ is equipped with the norm 
\[
\Vert f\Vert_{H^1_0(\Omega)}=\Vert \nabla f\Vert_{L^2(\Omega)}.
\] 
In the language of interpolation theory, 
\bq\label{inter:1}
D(\Lambda^\alpha)=[L^2(\Omega), D(-\Delta)]_{\frac \alpha 2}\quad\forall \alpha\in [0, 2].
\eq
Moreover, it is readily seen by virtue of the H\"older inequality  that
\bq\label{inter}
\Vert f\Vert_{\alpha, D}\le \Vert f\Vert_{\alpha_1, D}^\mu\Vert f\Vert_{\alpha_2, D}^{1-\mu}
\eq
provided $\alpha_1, \alpha_2\ge 0$, $\alpha=\mu \alpha_1+(1-\mu)\alpha_2$, and $\mu\in [0, 1]$.\\
As mentioned above,
\[
H^1_0(\Omega)= D(\Lambda)=[L^2(\Omega), D(-\Delta)]_{\mez},
\]
hence
\[
D(\Lambda^\alpha)=[L^2(\Omega), H^1_0(\Omega)]_{\alpha}\quad\forall \alpha\in [0, 1].
\]
Consequently, we can identify $D(\Lambda^\alpha)$ with usual Sobolev spaces (see Chapter 1 \cite{LioMag}):
\bq\label{identify}
D(\Lambda^\alpha)=
\begin{cases}
H^\alpha_0(\Omega) &\quad\text{if}~ \alpha\in (\mez, 1],\\
H^\mez_{00}(\Omega):=\{ u\in H^\mez_0(\Omega): u/\sqrt{d(x)}\in L^2(\Omega)\}&\quad\text{if}~ \alpha=\mez,\\
H^\alpha(\Omega) &\quad\text{if}~ \alpha\in [0, \mez),\\
\end{cases}
\eq
%Recall (see Theorem 11.1 Chapter 1, \cite{LioMag}) the classical result
%\bq\label{Hsmall}
%H^s(\Omega)=H^s_0(\Omega)\quad\forall s\in [0, \mez].
%\eq
We have the following relation between $D(\L^s)$ and $H^s(\Omega)$.
\begin{prop}\label{prop:inject}
The continuous embedding 
\bq\label{inject}
D(\L^\alpha)\subset H^\alpha(\Omega)
\eq
holds for all $\alpha\ge 0$.
\end{prop}
\begin{proof}
By interpolation, it suffices to prove \eqref{inject} for  $\alpha\in \{0, 1, 2,...\}$. The case $\alpha=0$ is obvious while the case $\alpha=1$ follows from \eqref{identify}. Assume by induction \eqref{inject} for $\alpha\le m$ with $m\ge 1$. Let $\tt\in D(\L^{m+1})$ then $f:=-\Delta \tt\in D(\L^{m-1})$ and thus $f\in H^{m-1}(\Omega)$ by the induction hypothesis. On the other hand, $\tt$ vanishes on the boundary $\partial\Omega$ in the trace sense because $\tt\in D(\L^1)=H^1_0(\Omega)$. Elliptic regularity then implies that $\tt\in H^{m+1}(\Omega)$ and 
\[
\Vert \tt\Vert_{H^{m+1}}\le C\Vert f\Vert_{H^{m-1}}\le C\Vert \Delta \tt\Vert_{m-1, D}=C\Vert \tt\Vert_{m+1, D}
\]
which is \eqref{inject} for $\alpha=m+1$.
\end{proof}
Below is the list of some notations used throughout this paper:
\begin{itemize}
\item $(\cdot, \cdot)$: the $L^2(\Omega)$ scalar product.
\item $\langle\cdot, \cdot\rangle_{X', X}$: the dual pairing between $X$ and its dual $X'$.
\item $\gamma_0(u)$: the trace of $u$ on $\partial\Omega$.
\item $\gamma(u)$: the trace of $u\cdot \nu$ on $\partial \Omega$ where $\nu$ is the outward unit normal to $\partial\Omega$.
%\item For any real number $r\in [1, \infty]$, $r':=\frac{r}{r-1}\in [1, \infty]$ denotes its conjugate exponent.
\end{itemize}
%%%%%%%%%%%%%%%%%%%
\section{Proof of Theorem \ref{main1}}\label{section:main1}
\subsection{Technical lemmas}
We start with an estimate for the Riesz transforms in Sobolev spaces.
\begin{lemm} If $\tt\in D(\L^r)$ with $r\ge 0$ then
\bq\label{Sobolev:R}
\Vert R_D\tt\Vert_{H^r(\Omega)}\le C\Vert  \tt\Vert_{r, D}.
\eq
\end{lemm}
\begin{proof}
Indeed, we have $R_D \tt=\nabla \psi$ with $\psi=\L^{-1}\tt\in D(\L^{r+1})$. It follows from \eqref{inject} that
\[
\Vert R_D\tt\Vert_{H^r(\Omega)}\le \Vert  \psi\Vert_{H^{r+1}(\Omega)}\le C\Vert  \psi\Vert_{r+1, D}=C\Vert \tt\Vert_{r, D}.
\]
 \end{proof}
The next lemma provides the key estimate needed for the proof of Theorem \ref{main1}.
\begin{lemm}
Let $\alpha\in (0, 1)$ and $\tt\in D(\L^{2+\alpha})$. Denote $u=R^\perp \tt$ and $p=\frac{2}{1-\alpha}$. There exists a positive constant $C=C(\alpha, p)$ such that
\bq\label{commutator}
\Vert [\Delta, u\cdot\nabla]\tt\Vert_{L^2(\Omega)}\le CBA^{\frac{2-\alpha}{2}}\Vert \tt\Vert_{L^2(\Omega)}^\frac{\alpha}{2}
\eq
where
\bq\label{AB}
A=\Vert \L^2 \tt\Vert_{L^2(\Omega)}=\Vert\tt\Vert_{2, D},\quad B=\Vert \Lambda^{2+\alpha} \tt\Vert_{L^2(\Omega)}=\Vert\tt\Vert_{2+\alpha, D}.
\eq
\end{lemm}
\begin{proof}
A direct computation gives
\bq\label{cmtt:formula}
[\Delta, u\cdot\nabla]\tt=\Delta u\cdot \nabla \tt+2\nabla u\cdot \nabla\nabla \tt
\eq
where 
\[
\nabla u\cdot \nabla\nabla \tt:=\partial_1 u^1\partial^2_{11}\tt+\partial_2u^1\partial^2_{21}\tt+\partial_1u^2\partial_{21}\tt+\partial_2u^2\partial_{22}\tt
\]
if
\[ u=(u^1, u^2).
\]
Using the facts that $\Delta$ commutes with the Riesz transforms, because it commutes with both $\nabla$ and $\L^{-1}$, the Riesz transforms are bounded in $L^r$ for all $r\in (1, \infty)$, a fact that holds for $C^1$ domains (see Theorem C in \cite{Shen}), together with \eqref{identify} we deduce
\bq\label{cm1}
\Vert \Delta u\Vert_{L^p}=\Vert R_D^\perp\Delta \tt\Vert_{L^p}\le C\Vert \Delta \tt\Vert_{L^p}\le C\Vert \Delta \tt\Vert_{H^{\alpha}}\le C\Vert \Delta \tt\Vert_{\alpha, D}=CB.
\eq
where the embedding $H^{\alpha}\subset L^p$ was used in the second inequality.\\
 Let $q=\frac{2}{\alpha}$ satisfy $\frac{1}{p}+\frac{1}{q}=\mez$. By the embeddings \eqref{inject}, $H^{1-\alpha}\subset L^q$ and interpolation we have
 \bq\label{cm2}
\Vert \nabla\tt\Vert_{L^q}\le C\Vert \tt\Vert_{H^{2-\alpha}}\le C\Vert \tt\Vert_{H^2}^{\frac{2-\alpha}{2}}\Vert \tt\Vert_{L^2}^{\frac \alpha 2}\le CA^{\frac{2-\alpha}{2}}\Vert \tt\Vert_{L^2}^{\frac \alpha 2}.
\eq
Let us note that $\tt\in D(\L^{2+\alpha})\subset D(\L^1)=H^1_0(\Omega)$, so $\tt$ vanishes on the boundary $\partial\Omega$ in the trace sense. Elliptic estimates in $L^p$ together with the embeddings $H^{\alpha}\subset L^p$ and \eqref{inject} imply
\[
\Vert \nabla\nabla\tt\Vert_{L^p}\le \Vert \tt\Vert_{W^{2, p}}\le C\Vert \Delta\tt\Vert_{L^p}\le C\Vert \Delta\tt\Vert_{H^\alpha}\le C\Vert \tt\Vert_{2+\alpha, D}.
\]
Thus,
\bq\label{cm3}
\Vert \nabla\nabla\tt\Vert_{L^p}\le CB.
\eq
Now regarding the term $\nabla u$ we first use the embedding $H^{1-\alpha}\subset L^q$ and the estimate \eqref{Sobolev:R} to have
\[
\Vert \nabla u\Vert_{L^q}\le  \Vert u\Vert_{H^{2-\alpha}}=\Vert R^\perp_D\tt\Vert_{H^{2-\alpha}}\le C\Vert\tt\Vert_{2-\alpha, D},
\]
and then by the interpolation inequality \eqref{inter}
\bq\label{cm4}
\Vert \nabla u\Vert_{L^q}\le CA^{\frac{2-\alpha}{2}}\Vert \tt\Vert_{L^2}^{\frac{\alpha}{2}}.
\eq
Finally, putting together \eqref{cm1}-\eqref{cm4} we arrive at \eqref{commutator} by using the H\"older inequality with exponents $p$ and $q$.
\end{proof}
We recall the following product rule (see Chapter 2, \cite{BCD}) in $\Rr^d$, $d\ge 1$,
\bq\label{product}
\Vert f_1f_2\Vert_{H^{s_1}(\Rr^d)}\le C\Vert f_1\Vert_{H^{s_1}(\Rr^d)}\Vert f_2\Vert_{H^{s_2}(\Rr^d)}
\eq
provided 
\[
s_1\le s_2,\quad s_1+s_2>0,\quad s_2>\frac d2.
\]
By extension, interpolation, and duality, \eqref{product} still holds in smooth bounded domains of $\Rr^d$.
\begin{lemm}\label{lemm:trace}
Let $\tt \in D(\L^2)$, $\psi\in H^1_0(\Omega)\cap H^r(\Omega)$, $r>2$, and $u=\nabla^\perp \psi$. Then $u\cdot\nabla \tt\in H^1_0(\Omega)$.
\end{lemm}
\begin{proof}
First, let us note that $\gamma_0(u)\in H^{r-\tdm}(\partial\Omega)$ and $\gamma_0(\nabla \tt)\in H^\mez(\partial\Omega)$. In particular, $\gamma_0(u)\cdot \gamma_0(\nabla\tt)$ is well defined in $H^\mez(\partial\Omega)$ by virtue of the product rule \eqref{product} for $\Omega$. Since $\psi\in H^1_0(\Omega)$, $\gamma_0(u)=\gamma_0(\nabla^\perp \psi)$ is tangent to the boundary, and since $\tt\in H^1_0(\Omega)$, $\gamma_0(\nabla\tt)$ is normal to the boundary. Therefore, $\gamma_0(u)\cdot \gamma_0(\nabla\tt)$ vanishes on the boundary. Because the mapping $H^{r-1}(\Omega)\times H^1(\Omega)\to H^1(\Omega)$ is continuous in view of \eqref{product}, $\gamma_0(u\cdot\nabla \tt)=\gamma_0(u)\cdot \gamma_0(\nabla\tt)=0$. For the same reason, we have $u\cdot\nabla\tt\in H^1(\Omega)$ and hence $u\cdot\nabla\tt\in H^1_0(\Omega)$.
\end{proof} 
\subsection{Uniqueness}\label{section:unique}
Let 
\[
\tt_j\in L^\infty\big([0, T]; D(\L^2)\big)\cap L^2\big([0, T]; D(\L^{2+\alpha})\big) ,\quad \alpha\in (0, 2),
\]
$j=1,2$, be two solutions of the inviscid SQG equation with the same initial data $\tt_0$. Then the difference $\tt=\tt_1-\tt_2$ solves
\bq\label{unique:eq}
\partial_t\tt+u\cdot\nabla\tt_1+u_2\cdot\nabla\tt+\k\L^{2\alpha}\tt=0,\quad \tt\vert_{t=0}=0.
\eq
Here, $u=R_D^\perp \tt$. Multiplying this equation by $\tt$, then integrating over $\Omega$ gives
\[
\mez\frac{d}{dt}\Vert \tt\Vert_{L^2(\Omega)}^2=-\int_\Omega \tt u_1\cdot\nabla \tt-\int_\Omega \tt u\nabla \tt_2-\k\int_\Omega \tt\L^{2\alpha}\tt.
\]
After integrating by parts, the last term is nonpositive, the first term vanishes because $u_1$ is divergence free. The middle term is bounded by
\[
\Vert \tt\Vert_{L^2(\Omega)}\Vert R_D^\perp \tt\Vert_{L^2(\Omega)}\Vert \nabla \tt_2\Vert_{L^\infty(\Omega)}\le C\Vert \tt\Vert_{L^2(\Omega)}^2\Vert \tt_2\Vert_{2+\alpha, D},
\]
where we used the embeddings $D(\L^{2+\alpha})\subset H^{2+\alpha}(\Omega)\subset W^{1, \infty}(\Omega)$. Because $\tt_2\in L^2\big([0, T]; D(\L^{2+\alpha}))$, the Gr\"onwall lemma concludes that $\tt=0$ on $[0, T]$, and thus $\tt_1=\tt_2$.
\subsection{Local existence}\label{section:local}
Let $\alpha\in (0, 2)$ and let $\tt_0\in D(\L^2)=H^2(\Omega)\cap H^1_0(\Omega)$ be an initial datum. We prove local existence of solutions using the Galerkin approximations. Denote by $\P_m$ the projection in $L^2$ onto the linear span $L^2_m$ of eigenfunctions $\{w_1,...,w_m\}$, {\it {\it i.e.}}
\[
\P_m f=\sum_{j=1}^mf_jw_j\quad\text{for}~f=\sum_{j=1}^\infty f_jw_j.
\]
It is readily seen that $\P_m$ commutes with $\L^s$ on $D(\L^s)$ for any $s\ge 0$.

The $m$th Galerkin approximation of \eqref{SQG} is the following ODE system in the finite dimensional space $\P_mL^2(\Omega)$:
\bq\label{Galerkin}
\begin{cases}
\dot \tt_m+\P_m(u_m\cdot\nabla\tt_m)+\k \Lambda^{2\alpha}\tt_m=0&\quad t>0,\\
\tt_m=P_m\tt_0&\quad t=0
\end{cases}
\eq
with $\tt_m(x, t)=\sum_{j=1}^m\tt_ j^{(m)} (t)w_j(x)$ and $u_m={R_D}^\perp\tt_m$ automatically satisfying $\cnx u_m=0$. Note that in general $u_m\notin L^2_m$.
The existence of solutions of \eqref{Galerkin} at fixed $m$ follows from the fact that this is an ODE:
\[
\frac{d\theta^{(m)}_l}{dt} + \sum_{j,k=1}^m\gamma^{(m)}_{jkl}\theta^{(m)}_j\theta^{(m)}_{k} + \k\lambda_l^\alpha\tt^{(m)}_l= 0
\label{galmode}
\]
with
\[
\gamma^{(m)}_{jkl} = \lambda_j^{-\frac{1}{2}}\int_{\Omega}\left(\na^{\perp}w_j\cdot\na w_k\right)w_ldx.
\]
Since $\P_m$ is self-adjoint in $L^2$, $u_m$ is divergence-free and $w_j$ vanishes at the boundary $\partial\Omega$, integrations by parts give
\[
\int_\Omega \tt_m\P_m(u_m\cdot\nabla \tt_m)dx=\int_\Omega \tt_m u_m\cdot\nabla \tt_mdx=0
\]
and
\[
\int_{\Omega}\L^{2\alpha}\tt_m\tt_mdx=\Vert\L^\alpha \tt_m\Vert_{L^2}^2.
\]
 It follows that 
 \bq\label{local:L2}
 \mez\frac{d}{dt}\Vert \tt_m\Vert^2_{L^2}+\k \Vert\L^\alpha \tt_m\Vert_{L^2}^2=0
 \eq
 and in particular, the $L^2$ norm of $\tt_m$ is bounded:
\[
\Vert \tt_m(\cdot, t)\Vert^2_{L^2(\Omega)}=\Vert \P_m\tt_0(\cdot, 0)\Vert^2_{L^2(\Omega)}\le \Vert \tt_0\Vert^2_{L^2(\Omega)}.
\]
This can be seen directly on the ODE because $\gamma^{(m)}_{jkl}$ is antisymmetric in $k,l$.
Therefore, the smooth solution $\tt_m$ of \eqref{Galerkin} exists globally. Observe that for the sake of global existence of \eqref{Galerkin}, the dissipative effect is not needed, {\it i.e.} $\k$ can be $0$. Obviously, $\tt_m(\cdot, t)\in D(\L^r)$ for all $r\ge 0$ and $t\ge 0$. According to Lemma \ref{lemm:trace}, $u_m\cdot\tt_m\in H^1_0(\Omega)$ which combined  with the fact that $\Delta (u\cdot \tt_m)\in L^2(\Omega)$ implies $u_m\cdot\tt_m\in D(-\Delta)$. Now applying $\L^2=-\Delta$ to \eqref{Galerkin} and noticing that $\L^2$ commutes with $\P_m$ on $D(\L^2)$ result in
\[
\partial_t(\L^2\tt_m)+\P_m\big( [\L^2, u_m\cdot \nabla]\tt_m\big)+\P_m \big(u_m\cdot\nabla(\L^2\tt_m)\big)+\k\Lambda^{2+2\alpha}\tt_m=0
\]
 Next, we take the scalar product with $\L^2 \tt_m$, use the commutator estimate \eqref{commutator},  and the fact that $\P_m$ is self-adjoint in $L^2$ to arrive at the differential inequality
\bq\label{diff:ineq}
\mez\frac{d}{dt} A_m^2+\k B_m^2\le CB_mA_m^{\frac{4-\alpha}{2}}\Vert \tt_m\Vert_{L^2}^\frac{\alpha}{2}\le CB_mA_m^2
\eq
where $A_m$ and $B_m$ are defined as in \eqref{AB} for $\tt_m$.  Then an application of the Young inequality allows us to hide $B_m$  on the right-hand side of \eqref{diff:ineq} and obtain
\bq\label{diff:ineq'}
\mez\frac{d}{dt} A_m^2+\frac{\k}{2} B_m^2\le \frac{C}{\k}A_m^4.
\eq
Ignoring $B_m$ and integrating \eqref{diff:ineq'} leads to 
\[
A_m^2(t)\le 2A_m^2(0)\quad\forall t\in [0, T_m] 
\]
with
\[
T_m:=\frac{\k}{2CA_m(0)^2}\ge T:=\frac{\k}{2CA(0)^2},\quad A(0)=\Vert\tt_0\Vert_{2,D}.
\]
In other words, $\tt_m$ is uniformly in $m$ bounded in $L^\infty([0, T]; D(\L^2))$. Using the equation we find that $\partial_t \tt_m$ is uniformly in $m$ bounded in $L^\infty([0, T]; L^2(\Omega))$.  The Aubin-Lions lemma (\cite{Lions}) then allows us to conclude the existence of a solution $\tt$ of \eqref{SQG} on $[0, T]$. Moreover, by integrating \eqref{diff:ineq'} we find that $\tt$ satisfies
\bq\label{reg:tt}
\tt\in L^\infty\big([0, T]; D(\L^2)\big)\cap L^2\big([0, T]; D(\L^{2+\alpha})\big).
\eq
\subsection{Global existence}
Let $\alpha\in (0, 2)$ and let $\tt_0\in D(\L^2)$ be an initial datum.  We reuse the notations of section \ref{section:local}. Recall from \eqref{diff:ineq} that
\bq\label{diff:ineq0}
\mez\frac{d}{dt} A_m^2+\k B_m^2 \le CB_mA_m^{\frac{4-\alpha}{2}}\Vert \tt_m\Vert_{L^2}^\frac{\alpha}{2}.
\eq
 It is readily seen by the interpolation inequality \eqref{inter} that
\[
A_m=\Vert \tt_m\Vert_{2, D}\le C\Vert \Lambda^{2+\alpha}\tt_m\Vert_{L^2}^{\frac{2}{2+\alpha}}\Vert \tt_m\Vert_{L^2}^{\frac{\alpha}{2+\alpha}}= CB_m^{\frac{2}{2+\alpha}}\Vert \tt_m\Vert_{L^2}^{\frac{\alpha}{2+\alpha}}.
\]
Consequently
\begin{align*}
B_mA_m^{\frac{4-\alpha}{2}}\Vert \tt_m\Vert_{L^2}^\frac{\alpha}{2}&=B_mA_m^{1-\alpha}\Vert \tt_m\Vert_{L^2}^\frac{\alpha}{2}A_m^{\frac{2+\alpha}{2}}\\
&\le CB_mA_m^{1-\alpha}\Vert \tt_m\Vert_{L^2}^\frac{\alpha}{2}B_m\Vert \tt\Vert_{L^2}^\frac{\alpha}{2}\\
&\le CB_m^2A_m
\end{align*}
and thus
\bq\label{diff:ineq3}
\frac{d}{dt} A_m^2+\k B_m^2\le CB_m^2\big(A_m-\frac{\k}{C}\big),\quad C=C(\alpha).
\eq
Integrating this leads to
\[
A_m^2(t)+\int_0^t \k B_m^2 ds\le A_m^2(0)+C\int_0^t B_m^2\big(A_m-\frac{\k}{C}\big)ds\quad\forall t\ge 0.
\]
 By a coninuity argument,  if 
\bq
A(0)=\Vert\tt_0\Vert_{2, D}< \frac{\k}{C}
\eq
then $A_m(t)\le \frac{\k}{C}$ for $t\ge 0$ and thus, in view of \eqref{diff:ineq3}, $A_m(t)\le A_0$ for $t\ge 0$. In other words, the $D(\L^2)$  norm of $\tt_m$ is uniformly in $m$ bounded over all finite time interval $[0, T]$. Using the equation, we deduce a uniform bound for $\partial_t \tt_m$ in $L^\infty([0, T]; L^2(\Omega))$. Passing to the limit $m\to \infty$ then can be done by virtue of the Aubin-Lions lemma (\cite{Lions}) on each finite time interval $[0, T]$. By uniqueness, we obtain a unique global solution.
%%%%%%%%%%%%%%%%
\section{Proof of Theorem \ref{main3}}\label{section:main3}
We first prove the following key estimate for the nonlinearity.
\begin{lemm}\label{lemm:nonlin}
Let $\alpha\in (\mez, 1]$, $\frac 1q\in (0, \alpha-\mez)$, $s\in [\alpha, \alpha+1]$. Fix $\delta\in (0, \mez(\alpha-\mez-\frac 1q))$ and put
\[
N=
\begin{cases}
\frac{\alpha}{\alpha-\mez-\frac 1q}\quad&\text{if}~s\ne \mez+\alpha,\\
\frac{\alpha}{\alpha-\delta-\mez-\frac 1q}\quad&\text{if}~s=\mez+\alpha.
\end{cases}
\]
Then with $\tt\in D(\L^2)$ and $u=R_D^\perp \tt$ we have for all $\eps>0$
\bq\label{sub:nonlin}
\la \int _{\Omega}\L^{s+\alpha}\tt \L^{s-\alpha}(u\cdot \nabla \tt)dx\ra\le 3\eps\Vert \tt\Vert^2_{s+\alpha, D}+\eps\Vert u\Vert^2_{H^{s+\alpha}}+C_\eps\Vert u\Vert_{L^q}^N\Vert \tt\Vert_{H^s}^2+C_\eps\Vert \tt\Vert_{L^q}^N\Vert u\Vert_{H^s}^2.
\eq
\end{lemm}
\begin{proof}
According to Lemma \ref{lemm:trace}, $u\cdot\nabla \tt\in D(\L)$. Let $p$ satisfy $\frac 1p+\frac 1q=\mez$ and put
\[
\beta=
\begin{cases}
1+\frac 2q-\alpha\quad&\text{if}~s\ne \mez+\alpha,\\
1+\frac 2q-\alpha+\delta\quad&\text{if}~s=\mez+\alpha.
\end{cases}
\]
Note that $\beta\in (0, \alpha)$ and $N=\frac{2\alpha}{\alpha-\beta}$ is the conjugate exponent of $\frac{2\alpha}{\alpha+\beta}$, {\it i.e.} $\frac{1}{N}+\frac{\alpha+\beta}{2\alpha}=1$.  By  \eqref{identify}, $D(\L^{s-\alpha})=H^{s-\alpha}_0(\Omega)$ if $s-\alpha\ne \mez$ and $H^{s-\alpha+\delta}_0(\Omega)\subset D(\L^{s-\alpha})$ if $s-\alpha=\mez$. Writing $u\cdot \nabla\tt=\cnx (u\tt)$ we estimate using the H\"older inequality
\[
I:=\la \int _{\Omega}\L^{s+\alpha}\tt \L^{s-\alpha}(u\cdot \nabla \tt)dx\ra\le \Vert \tt\Vert_{s+\alpha, D}\Vert \cnx (u\tt)\Vert_{H^{s-\alpha}}\le \Vert \tt\Vert_{s+\alpha, D}\Vert u\tt\Vert_{H^{s+1-\alpha}}
\]
if $s-\alpha\ne \mez$, and similarly,
\[
I\le \Vert \tt\Vert_{s+\alpha, D}\Vert u\tt\Vert_{H^{s+1-\alpha+\delta}}
\]
if $s-\alpha=\mez$.

In $\Rr^d$ we have
\begin{align*}
\Vert \phi_1\phi_2\Vert_{H^{s+1-\alpha}}&\le C\Vert \phi_1\Vert_{L^q}\Vert  \phi_2\Vert_{W^{s+1-\alpha, p}}+C\Vert \phi_2\Vert_{L^q}\Vert  \phi_1\Vert_{W^{s+1-\alpha, p}}\\
&\le C\Vert \phi_1\Vert_{L^q}\Vert  \phi_2\Vert_{H^{s+\beta}}+C\Vert \phi_2\Vert_{L^q}\Vert  \phi_1\Vert_{H^{s+\beta}}
\end{align*}
in view of the embedding $H^{s+\beta}(\Rr^d)\subset W^{s+1-\alpha, p}(\Rr^d)$. Then by extension  and interpolation the following inequality holds in $\Omega$
\[
\Vert \phi_1\phi_2\Vert_{H^{s+1-\alpha}}\le C\Vert \phi_1\Vert_{L^q}\Vert  \phi_2\Vert_{H^{s+\beta}}+C\Vert \phi_2\Vert_{L^q}\Vert  \phi_1\Vert_{H^{s+\beta}}
\]
which implies
\[
 \Vert u\tt\Vert_{H^{s+1-\alpha}}\le C\Vert u\Vert_{L^q}\Vert  \tt\Vert_{H^{s+\beta}}+C\Vert \tt\Vert_{L^q}\Vert  u\Vert_{H^{s+\beta}}.
\]
The same estimate holds with $\alpha$ replaced with $\alpha-\delta$. We thus obtain in both cases
\[
I\le C\Vert \tt\Vert_{s+\alpha, D}\Vert u\Vert_{L^q}\Vert  \tt\Vert_{H^{s+\beta}}+C\Vert \tt\Vert_{s+\alpha, D}\Vert \tt\Vert_{L^q}\Vert  u\Vert_{H^{s+\beta}}.
\]
By interpolation, we have
\[
\Vert  \phi\Vert_{H^{s+\beta}}\le \Vert  \phi\Vert_{H^{s+\alpha}}^{\frac{\beta}{\alpha}}\Vert  \phi\Vert_{H^s}^{\frac{\alpha-\beta}{\alpha}}.
\]
 Applying Young inequalities yields for all $\eps>0$
\begin{align*}
\Vert \tt\Vert_{s+\alpha, D}\Vert u\Vert_{L^q}\Vert  \tt\Vert_{H^{s+\beta}}&\le \eps\Vert \tt\Vert_{s+\alpha, D}^{\frac{2\alpha}{\alpha+\beta}}\Vert\tt\Vert_{H^{s+\alpha}}^{\frac{2\beta}{\alpha+\beta}}+C_\eps\big( \Vert u\Vert_{L^q}\Vert \tt\Vert_{H^\alpha}^\frac{\alpha-\beta}{s}\big)^N\\
&=\eps\Vert \tt\Vert_{s+\alpha, D}^{\frac{2\alpha}{\alpha+\beta}}\Vert\tt\Vert_{H^{s+\alpha}}^{\frac{2\beta}{\alpha+\beta}}+C_\eps \Vert u\Vert_{L^q}^N\Vert \tt\Vert_{H^s}^2\\
&\le \eps\Vert \tt\Vert_{s+\alpha, D}^2+\eps\Vert \tt\Vert_{H^{s+\alpha}}^2+C_\eps \Vert u\Vert_{L^q}^N\Vert \tt\Vert_{H^s}^2
\end{align*}
and similarly,
\[
\Vert \tt\Vert_{s+\alpha, D}\Vert \tt\Vert_{L^q}\Vert  u\Vert_{H^{s+\beta}}\le \eps\Vert \tt\Vert_{s+\alpha, D}^2+\eps\Vert u\Vert_{H^{s+\alpha}}^2+C_\eps \Vert \tt\Vert_{L^q}^N\Vert u\Vert_{H^s}^2.
\]
Using the embedding $D(\L^{s+\alpha})\subset H^{s+\alpha}$ and putting together the above considerations leads to the estimate \eqref{sub:nonlin}.
\end{proof}
\begin{rema}
When $\Omega=\Rr^2, \Tt^2$, the estimate \eqref{sub:nonlin} holds for any $s>0$ (see Chapter 3 \cite{Res}). Here, for domains with boundaries, the restriction $s\le 1+\alpha$ was imposed because  $s-\alpha>1$ requires more vanishing conditions for $u\cdot \nabla\tt$ on $\partial\Omega$  in order to have $u\cdot\nabla\tt\in D(\L^{s-\alpha})$. In addition, product rules for $\L^\beta (ab)$ with $\beta >1$ are not available. In the above proof, the fact that $s-\alpha\le 1$ helped bounding $\Vert \L^\beta (ab)\Vert_{L^2}$ by $\Vert ab\Vert_{H^\beta}$, in view of \eqref{identify}, and then we could use the product rules in usual Sobolev spaces.

The restriction $s\le 1+\alpha$ at first limits the regularity of the solution, {\it i.e.} $\tt\in L^\infty_t D(\L^{1+\alpha})\cap L^2_tD(\L^{1+2\alpha})$. In order to gain the full regularity $ L^\infty_tD(\L^2)\cap L^2_tD(\L^{2+\alpha})$ we note that $u=R_D^\perp \tt\in L^2_tD(\L^{1+2\alpha})\subset L^2_tW^{2, q}$ with $q>2$ because $2\alpha>1$. Then, using the result of Appendix 2, we know that in general the linear transport equation
\[
\partial_t f+u\cdot \nabla f+\k \L^{2\alpha} f=0
\]
has a solution $f\in L^\infty_tD(\L^2)\cap L^2_tD(\L^{2+\alpha})$. Moreover, uniqueness holds in the class of $f\in L^\infty_t(H^1_0\cap L^\infty)$. The known regularity of $\tt$ is thus enough to conclude that $\tt=f$, and thus $\tt$ has the full regularity. The rest of this section is devoted to implement this strategy.
\end{rema}
Let $\tt_0\in D(\L^2)$ be an initial datum and $T>0$ be fixed. We construct a solution for \eqref{SQG} using the retarded mollifications. To this end we pick a $\phi\in C^\infty((0, \infty))$, $\phi\ge 0$, with $\supp\phi\in [1, 2]$, and let 
\[
U_\delta[\tt](t)=\int_0^\infty \phi(\tau)R_D^\perp\tt(t-\delta\tau) d\tau
\]
where we set $\tt(t)=0$ for all $t<0$. In particular, $U_\delta[\tt](t)$ depends on the values of $\tt(t')$ only for $t'\in [t-2\delta, t-\delta]$.

{\bf Step 1.} We pick a sequence $\delta_m\to 0^+$ and consider the approximate equations for $\tt_m$
\bq\label{retard:eq}
\partial_t\tt_m+u_m\cdot \nabla\tt_m+\k\L^{2\alpha}\tt_m=0
\eq
with initial data $\tt_m(0)=\tt_0$ and velocity $u_m:=U_{\delta_m}[\tt_m]$. For a fixed $m$, equation \eqref{retard:eq} is linear on each subinterval $I_k:=[t_k, t_{k+1}]$, $t_k:=k\delta_m$, $k\in \Zz$, because $u_m$ is determined by the values of $\tt_m$ on the two previous subintervals $I_{k-1}$ and $I_{k-2}$. By our setting, $\tt_m\equiv 0$ on $\cup_{k< 0} I_k$. On $I_0$, $u_m=0$ and the linear equation \eqref{retard:eq} with initial data $\tt_m(0)=\tt_0$ has a unique solution 
\[
\tt_m(t)=\sum_{j\ge 1}e^{-\lambda_j^\alpha t}\tt_{0, j}w_j\quad \text{with}~\tt_{0, j}=\int_\Omega \tt_0w_j dx.
\]
Direct estimates show that
\[
\tt_m\in L^\infty\big(I_0; D(\L^2)\big)\cap L^2\big(I_0; D(\L^{2+\alpha})\big).
\]
This implies in view of \eqref{Sobolev:R} that
\[
u_m\in L^2(I_1; H^{2+\alpha})\subset W^{2, p}
\]
with  $p=\frac{2}{1-\alpha}>2$. This regularity of $u_m$ on $I_1$ suffices to conclude by applying Theorem 4 in \cite{ConIgn} that there exists a unique solution $\tt_m$ on $I_1$ and thus, by induction, on $I_k$ for all $k\ge 1$, and
\[
\tt_m\in L^\infty\big(I_k; D(\L^2)\big)\cap L^2\big(I_k; D(\L^{2+\alpha})\big).
\]
 The proof of Theorem 4 in \cite{ConIgn} makes use of a general commutator estimate for $[\L, u\cdot \nabla]\tt$ in $D(\L^\mez)$ derived in the same paper. In Appendix 2, we give a direct proof without the commutator estimate. 
 
We showed so far that for any fixed integer $m$, equation \eqref{retard:eq} with initial data $\tt_0$ has a solution
\bq\label{reg:app}
\tt_m\in L^\infty\big([0, T]; D(\L^2)\big)\cap L^2\big([0, T]; D(\L^{2+\alpha})\big).
\eq
{\bf Step 2.} We appeal to Lemma \ref{lemm:nonlin} to pass to the limit $m\to \infty$ in the larger space $D(\L^{\alpha+1})$. First, it follows from \eqref{Sobolev:R}, \eqref{reg:app}, and the definition of $u_m$ that
\bq\label{sub:bound1}
\int_0^t\Vert u_m(\tau)\Vert_{H^{r}}^2\le C\int_0^t\Vert\tt_m(\tau)\Vert_{r, D}^2 d\tau,\quad t\in [0, T],~r\in [0, 2+\alpha].
\eq	
Secondly, according to Proposition \ref{prop:Lr}, the $L^r$ bounds
\bq\label{sub:bound2}
\sup_{[0, t]}\Vert u_m(\tau)\Vert_{L^r}\le C\sup_{[0, t]}\Vert\tt_m(\tau)\Vert_{L^r}\le C\Vert\tt_0\Vert_{L^r},\quad t\in [0, T],
\eq
hold for all $r\ge 4$.

Let us fix $s=\alpha+1$ and
\[
q>\min\left\{4, \big(\alpha-\mez\big)^{-1}\right\}.
\]
 Applying  $\L^{s-\alpha}$ in \eqref{retard:eq}, then taking the scalar product with $\L^{s+\alpha}\tt_m$ we obtain
 \[
\mez\frac{d}{dt}\Vert\tt_m\Vert_{s, D}^2+\k\Vert\tt_m\Vert_{s+\alpha, D}^2= \la\int_\Omega \L^{s+\alpha}\tt_m\L^{s-\alpha}(u_m\nabla \tt_m) dx\ra.
\]
Using \eqref{sub:nonlin} (note that $\tt_m\in D(\L^2)$) to estimate the right-hand side and then integrating the differential inequality we obtain for $t\le T$
\begin{align*}
&\Vert\tt_m(t)\Vert_{s, D}^2+2\k\int_0^t\Vert\tt_m(\tau)\Vert^2_{s+\alpha, D}d\tau\\
&\le \Vert\tt_0\Vert_{\alpha, D}^2+6\eps\int_0^t\Vert \tt_m(\tau)\Vert^2_{s+\alpha, D}d\tau+2\eps\int_0^t\Vert u_m(\tau)\Vert^2_{H^{s+\alpha}}d\tau\\
&\quad+C_\eps\int_0^t\Vert u_m(\tau)\Vert_{L^q}^N\Vert \tt_m(\tau)\Vert_{H^s}^2d\tau+C_\eps\int_0^t\Vert \tt_m(\tau)\Vert_{L^q}^N\Vert u_m(\tau)\Vert_{H^s}^2d\tau.
\end{align*}
We choose $\eps=\frac{\k}{M}$, $M$ being sufficiently large,  use \eqref{sub:bound1}, \eqref{sub:bound2}, \eqref{inject} and the Gr\"onwall lemma to arrive at
\bq\label{unibound:sub}
\Vert\tt_m\Vert_{L^\infty([0, T]; D(\L^s))}+\Vert\tt_m\Vert_{L^2([0, T]; D(\L^{s+\alpha}))} \le C\Vert\tt_0\Vert_{\alpha, D}\exp\big(CT\Vert\tt_0\Vert_{L^q}^N\big)
\eq
with $C=C(\k)$. The use of equation \eqref{retard:eq} and the bound \eqref{sub:bound1} implies that $\partial_t\tt_m$ is uniformly in $m$ bounded in $L^2([0, T]; L^2(\Omega))$. The Aubin-Lions lemma (\cite{Lions}) then allows us to conclude the existence of a solution 
\[
\tt\in L^\infty\big([0, T]; D(\L^s)\big)\cap L^2\big([0, T]; D(\L^{s+\alpha})\big)
\]
of \eqref{SQG}. Moreover, $\tt$ obeys the bound \eqref{unibound:sub}.

We note that $u=R^\perp_D\tt\in L^2([0, T]; H^{s+\alpha}(\Omega))$ with $s+\alpha=1+2\alpha>2$ and hence $u\in L^2([0, T]; W^{2, p}(\Omega))$ with $p=\frac{2}{1-\alpha}>2$. According to Theorem \ref{wp:ade} 1., there exists a solution 
\bq\label{reg:tt1}
\tt_1\in L^\infty\big([0, T]; D(\L^2)\big)\cap L^2\big([0, T]; D(\L^{2+\alpha})\big)
\eq
of the linear equation
\[
\partial_t\tt_1+u\cdot\nabla \tt_1+\L^{2\alpha}\tt_1=0,\quad\tt_1\arrowvert_{t=0}=\tt_0\in D(\L^2).
\]
The regularity of $\tt$ is sufficient to conclude using Theorem \ref{wp:ade} 2. that $\tt=\tt_1$ and thus $\tt$ has the full regularity as in \eqref{reg:tt1}. Uniqueness follows from section \ref{section:unique}.
%%%%%%%%%%%%%%%%%%%%%%%
 \section{Proof of Theorem \ref{wp:inviscid}}\label{section:inviscid}
Let $\tt_0\in H^1_0(\Omega)\cap W^{2, p}(\Omega)$ with $p\in (2, \infty)$. The proof proceeds by Picard's iterations in each of which a viscosity approximation is added: $\tt_n$, $n\ge 1$, is defined as the solution of the problem
\bq\label{invis:tt0}
\begin{cases}
\partial_t\tt_n+u_n\cdot\nabla\tt_n-\k\Delta\tt_n=0, \quad (x, t)\in \Omega\times(0, \infty),\quad\k>0,\\
u_n=R_D^\perp \tt_{n-1},\\
\tt_n\vert_{t=0}=\tt_0.
\end{cases}
\eq
We prove by induction that there exist 
\[
T_0=T_0(\Vert \tt_0\Vert_{H^1_0\cap W^{2, p}}, p)>0, \quad M_0=M_0(\Vert \tt_0\Vert_{H^1_0\cap W^{2, p}}, p)>0,
\]
 both are independent of $n$ and $\k$, such that 
\bq\label{invis:1}
\tt_n\in L^\infty([0, T_0]; H^1_0(\Omega)\cap W^{2, p}(\Omega))
\eq
and
\bq\label{invis:2}
\Vert \tt_n\Vert_{L^\infty([0, T_0]; W^{2, p}(\Omega))}\le M_0.
\eq
When $n=0$, both \eqref{invis:1} and \eqref{invis:2} hold for any $T_0>0$. Assume they hold for $n\le k-1$, $k\ge 1$, we prove it for $n=k$. The regularity \eqref{invis:1} of $\tt_k$ will be obtained by three bootstraps: $H^2$, then $W^{2, q}$ with $q\in (2, p)$, and finally $W^{2, p}$.

{\bf Step 1.} $H^2$ regularity. We note that $\Delta u_k=R_D^\perp\Delta \tt_{k-1}\in L^p(\Omega)$. On the other hand, by Sobolev's embedding $\tt_{k-1}\in C^{1, \gamma}(\overline\Omega)$ for some $\gamma>0$, and $\gamma_0(\tt_{k-1})=0$, Proposition 3.1 \cite{CabTan} then yields $\L^{-1}\tt_{k-1}\in C^{2, \gamma}(\overline\Omega)$, and thus $u_k\in C^{1, \alpha}(\overline\Omega)\subset W^{1, \infty}(\Omega)$.  Thus, 
\bq\label{est:u1}
\Vert \Delta u_k\Vert_{L^p(\Omega)}+\Vert u_k\Vert_{W^{1, \infty}(\Omega)}\le C\Vert \tt_{k-1}\Vert_{W^{2, p}(\Omega)}.
\eq
Note however that we do not have $u_k\in W^{2, p}(\Omega)$ in general but only $u_k\in W^{2, p}_{loc}(\Omega)$, by interior elliptic estimates. Then according to Theorem \ref{wp:ade}, the transport problem \eqref{invis:tt0} has a unique solution
\[
\tt_k\in L^\infty([0, T]; D(\L^2))\cap L^2([0, T]; D(\L^4))
\]
for any $T>0$ and 
\bq\label{est:tt1}
\begin{aligned}
\Vert \tt_k\Vert_{L^\infty([0, T]; D(\L^2))}+\k\Vert \tt_k\Vert_{L^2([0, T]; D(\L^{2+\alpha}))}&\le C\Vert \tt_0\Vert_{2, D}\exp\big(C\Vert \tt_{k-1}\Vert_{L^1([0, T]; W^{2, p})}\big)\\
&\le C\Vert \tt_0\Vert_{2, D}\exp\big(CT\Vert \tt_{k-1}\Vert_{L^\infty([0, T]; W^{2, p})}\big).
\end{aligned}
\eq
{\bf Step 2.} $W^{2, q}$ regularity.  Fix $q\in (2, p)$. We observe that $w_k=\Delta \tt_k$ satisfies
\bq\label{eq:w1}
\partial_t w_k+u_k\cdot\nabla w_k-\k\Delta w_k=-\Delta u_1\nabla\tt_k-2\nabla u_k\cdot\nabla\nabla\tt_k.
\eq
It follows from \eqref{est:u1}, \eqref{est:tt1}, and the embeddings $D(\L^2)\subset H^2(\Omega)\subset W^{1, r}(\Omega)$ for any $r<\infty$, that
\begin{align}\label{source:w1}
\Vert \Delta u_k\nabla\tt_k\Vert_{L^q(\Omega)}&\le \Vert \Delta u_k\Vert_{L^p(\Omega)}\Vert \nabla\tt_k\Vert_{L^r(\Omega)}\\
&\le C\Vert \tt_{k-1}\Vert_{W^{2, p}(\Omega)}\Vert \tt_0\Vert_{2, D}\exp\big(CT\Vert \tt_{k-1}\Vert_{L^\infty([0, T]; W^{2, p})}\big),
\end{align}
here $\frac 1q=\frac 1p+\frac 1r$.

In addition, because $\gamma_0(\tt_k)=0$ and $\tt_k\in D(\L^4)\subset H^4(\Omega)$, elliptic estimates combined with \eqref{est:tt1} imply
\bq\label{ell:w1}
\Vert \nabla\nabla\tt_k\Vert_{L^q(\Omega)}\le \Vert \tt_k\Vert_{W^{2, q}(\Omega)}\le C\Vert \Delta\tt_k\Vert_{L^q(\Omega)}=C\Vert w_k\Vert_{L^q(\Omega)}.
\eq
Now we multiply \eqref{eq:w1} by $q|w_k|^{q-2}w_k$, using the inequality \eqref{CorCor}, the fact that $\cn u_k=0$, and \eqref{ell:w1} to get
\begin{align*}
\frac{d}{dt}\Vert w_k\Vert^q_{L^q_x}&\le q \Vert \Delta u_k\nabla\tt_k\Vert_{L^q_x}\Vert w_k\Vert_{L^q_x}^{q-1}+2q\Vert \nabla u_k\Vert_{L^\infty_x}\Vert \nabla\nabla\tt_k\Vert_{L^q_x}\Vert w_k\Vert_{L^q_x}^{q-1}\\
&\le q \Vert \Delta u_k\nabla\tt_k\Vert_{L^q_x}\Vert w_k\Vert_{L^q_x}^{q-1}+qC\Vert \nabla u_k\Vert_{L^\infty_x}\Vert w_k\Vert_{L^q_x}^q.
\end{align*}
Consequently, for any $T>0$,
\begin{align*}
&\Vert w_k\Vert_{L^\infty([0, T]; L^q)}\\
&\le C\big(\Vert w_k(0)\Vert_{L^q}+\Vert \Delta u_k\nabla\tt_k\Vert_{L^1([0, T]; L^q)}\big)\exp\big(C\Vert \nabla u_k\Vert_{L^1([0, T]; L^\infty)} \big)\\
&\le \left(\Vert \tt_0\Vert_{W^{2,q}}+CT\Vert \tt_{k-1}\Vert_{L^\infty([0, T]; W^{2, p})}\Vert \tt_0\Vert_{2, D}\exp\big(CT\Vert \tt_{k-1}\Vert_{L^\infty([0, T]; W^{2, p})}\big)\right)\exp\big(CT\Vert \tt_{k-1}\Vert_{L^\infty([0, T]; W^{2, p})}\big)\\
&\le \cF(\Vert \tt_0\Vert_{W^{2,q}}+T\Vert \tt_{k-1}\Vert_{L^\infty([0, T]; W^{2, p})})
\end{align*}
for some increasing function $\cF:\Rr^+\to \Rr^+$, where \eqref{source:w1}, \eqref{est:u1} were used. In what follows, $\cF$ may change from line to line but is independent of $k$ and $\kappa$.\\
As in \eqref{ell:w1}, elliptic estimates yield
\[
\Vert \tt_k\Vert_{L^\infty([0, T]; W^{2, q})}\le C\Vert w_k\Vert_{L^\infty([0, T]; L^q)}\le \cF(\Vert \tt_0\Vert_{W^{2,q}}+T\Vert \tt_{k-1}\Vert_{W^{2, q}}).
\]
{\bf Step 3.} $W^{2, p}$ regularity. By the Sobolev embedding $W^{2, q}(\Omega)\subset W^{1, \infty}(\Omega)$, we have
\[
\Vert \tt_k\Vert_{L^\infty([0, T]; W^{1, \infty})}\le \cF(\Vert \tt_0\Vert_{W^{2, q}}+T\Vert \tt_{k-1}\Vert_{L^\infty([0, T]; W^{2, p})})
\]
which, combined with \eqref{est:u1}, implies 
\begin{align*}
\Vert \Delta u_k\nabla\tt_k\Vert_{L^\infty([0, T]; L^p)}&\le \Vert \Delta u_k\Vert_{L^\infty([0, T]; L^p)}\Vert \nabla\tt_k\Vert_{L^\infty([0, T]; L^\infty)}\\
&\le C\Vert \tt_{k-1}\Vert_{L^\infty([0, T]; W^{2, p})}\cF(\Vert \tt_0\Vert_{W^{2,q}}+T\Vert \tt_{k-1}\Vert_{L^\infty([0, T]; W^{2, p})}).
\end{align*}
Then, multiplying \eqref{eq:w1} by $p|w_k|^{p-2}w_k$ and argue as above leads to the $L^p$ bound
\begin{align*}
\Vert w_k\Vert_{L^\infty([0, T]; L^p)}&\le C\big(\Vert w_k(0)\Vert_{L^p}+\Vert \Delta u_k\nabla\tt_k\Vert_{L^1([0, T]; L^p)}\big)\exp\big(\Vert \nabla u_k\Vert_{L^1([0, T]; L^\infty)} \big)\\
& \le \cF(\Vert \tt_0\Vert_{W^{2, p}(\Omega)}+T\Vert \tt_{k-1}\Vert_{L^\infty([0, T]; W^{2, p})}).
\end{align*}
By elliptic estimates, we obtain that
\[
\Vert \tt_k\Vert_{L^\infty([0, T]; W^{2, p})}\le  \cF(\Vert \tt_0\Vert_{W^{2, p}(\Omega)}+T\Vert \tt_{k-1}\Vert_{L^\infty([0, T]; W^{2, p})}).
\]
{\bf Step 4.} Concluding. Now by the induction hypothesis, 
\[
\Vert \tt_{k-1}\Vert_{L^\infty([0, T_0]; W^{2, p})}\le M_0,
\]
with $T_0=T_0(\Vert \tt_0\Vert_{H^1_0\cap W^{2, p}}, p)>0$, $M_0=M_0(\Vert \tt_0\Vert_{H^1_0\cap W^{2, p}}, p)>0$ . Therefore, if we choose 
\[
M_0\ge \cF(2\Vert \tt_0\Vert_{W^{2, p}(\Omega)}),\quad T_0\le \frac{\Vert \tt_0\Vert_{W^{2, p}(\Omega)}}{M_0}\le \frac{\Vert \tt_0\Vert_{W^{2, p}(\Omega)}}{\cF(2\Vert \tt_0\Vert_{W^{2, p}(\Omega)})}
\]
then
\[
\cF\big(\Vert \tt_0\Vert_{W^{2, p}(\Omega)}+T_0M_0\big)\le M_0,
\]
and thus 
\bq\label{ubi:1}
\Vert \tt_k\Vert_{L^\infty([0, T_0]; W^{2, p})}\le M_0.
\eq
This completes the proof of \eqref{invis:1} and \eqref{invis:2}. Then, using the first equation in \eqref{invis:tt0}, \eqref{est:u1}, \eqref{est:tt1}, it follows easily that
\bq\label{ubi:2}
\Vert \partial_t\tt_n\Vert_{L^\infty([0, T_0]; L^2)}\le M_1
\eq
for some $M_1>0$ independent of $n$ and $\kappa$.

Using the uniform bounds \eqref{invis:2}, \eqref{ubi:2}, we can first pass to the limit $n\to 0$ by virtue of the Aubin-Lions lemma, then send $\k\to 0$ to obtain a solution 
\[
\tt\in L^\infty([0, T_0]; H^1_0(\Omega)\cap W^{2, p}(\Omega))
\]
to the inviscid SQG equation. Finally, uniqueness follows easily by an $L^2$ energy estimate for the difference of two solutions as done in section \ref{section:unique}, noticing that $\nabla\tt\in L^\infty_tW^{1, p}_x\subset L^\infty_{t,x}$ with $p>2$.
%%%%%%%%%%%%%%%%%%%%%%
\section{Appendix 1: $L^p$ bounds}
Let $\Omega\subset \Rr^2$ be an open set with smooth boundary.
\begin{prop}\label{prop:Lr}
Let $\alpha\in (0, 1]$ and $\k> 0$. Let $u\in L^\infty([0, T]; L^2(\Omega)^2)$ be a divergence-free vector field and consider the linear advection-diffusion equation
\bq\label{eq:Lp}
\partial_t\tt+u\cdot\nabla \tt+\k\L^{2\alpha}\tt=0,\quad\tt\vert_{t=0}=\tt_0.
\eq
(i) If $\alpha\in (\mez, 1)$ and 
\bq\label{Lp:reg}
\tt\in L^\infty\big([0, T]; D(\L^2)\big)\cap L^2\big([0, T]; D(\L^{2+\alpha})\big)
\eq
is a solution of \eqref{eq:Lp} then we have for any $r\in [4, \infty]$
\bq\label{app:Lr}
\Vert \tt\Vert_{L^\infty([0, T]; L^r(\Omega))}\le \Vert \tt_0\Vert_{L^r(\Omega)}.
\eq
(ii) If $\alpha\in (0, \mez]$ and 
\bq\label{Lp:reg:sub}
\tt\in L^\infty\big([0, T]; D(\L^2)\big)
\eq
is a solution of \eqref{eq:Lp} then \eqref{app:Lr} holds for any $r\in [2, \infty]$.
\end{prop}
\begin{proof}
We first note that in both cases, equation \eqref{eq:Lp} is satisfied in $L^2([0, T]; L^r(\Omega))$ for any $r\in [1, \infty]$. Therefore, $\tt\in C([0, T]; L^r(\Omega))$ for any $r\in [1, \infty]$.

(i) Case 1: $\alpha\in (\mez, 1)$ and  $r\in [4, \infty]$. It suffices to consider $r\in [4, \infty)$ because the case $r=\infty$ follows by sending $r\to \infty$. We have 
\[
\frac{d}{dt}\Vert \tt\Vert_{L^r}^r=\int_\Omega r|\tt|^{r-2}\tt\partial_t\tt=-\int_\Omega  u\cdot\nabla|\tt|^r dx-\k\int_\Omega r|\tt|^{r-2}\tt \L^{2\alpha}\tt dx. 
\]
In two dimensions, the condition $\tt\in D(\L^2)$ implies $|\tt|^r\in H^1_0(\Omega)$. Since $u$ is divergence-free, the first term on the right-hand side vanishes in view of the Stokes formula. Regarding the dissipative term, we use the C\'ordoba-C\'ordoba inequality (\cite{CorCor}, see also \cite{Res}) which was proved for bounded domains in (\cite{ConIgn}):
\bq\label{CorCor}
\Phi'(f)\L^sf-\L^s(\Phi(f))\ge 0,\quad s\in [0, 2],
\eq
almost everywhere in $\Omega\subset \Rr^2$ for $f\in H^1_0(\Omega)\cap H^2(\Omega)$ and $C^2(\Rr)$ convex $\Phi$ satisfying $\Phi(0)=0$. Note that in two dimensions, $f\in L^\infty(\Omega)$ and $\Phi(f)\in H^1_0(\Omega)\cap H^2(\Omega)$, hence each term in \eqref{CorCor} is well defined in $L^2(\Omega)$. Under condition \eqref{Lp:reg}, with $\Phi(z)=|z|^m\in C^2$, $m=\frac r2\ge 2$, we have
\begin{align*}
\int_\Omega r|\tt|^{r-2}\tt \L^{2\alpha}\tt dx&= 2\int_\Omega |\tt|^m m|\tt|^{m-2}\tt\L^{2\alpha}\tt dx\\
&\ge 2\int_\Omega |\tt|^m \L^{2\alpha}|\tt|^m dx\\
&=2\int_\Omega \vert \L^{\alpha}|\tt|^m\vert^2 dx\ge 0.
\end{align*}
Consequently $\frac{d}{dt}\Vert \tt\Vert_{L^r}^r\le 0$ and \eqref{app:Lr} follows. 

(ii) Case 2: $\alpha\in (0, \mez]$ and $r\in [2, \infty]$. If $s\in [0, 1]$ it suffices to assume $f\in H^1_0(\Omega)\cap H^s(\Omega)$ with $s>1$ and $\Phi\in C^1(\Rr)$ convex to get the inequality \eqref{CorCor}. Indeed, we then have $\Phi(f)\in H^1_0(\Omega)=D(\L^1)$ and thus $\L^s(\Phi(f))$ belongs to $L^2(\Omega)$. Therefore, \eqref{app:Lr} holds for any $r\ge 2$ by choosing $\Phi(z)=|z|^{\frac r2}\in C^1$ as in (i). 
\end{proof}
%%%%%%%%%%%%%
\section{Appendix 2: Linear advection-difussion}
Let $\Omega\subset \Rr^d$, $d\ge 2$, be an open set with smooth boundary. Let $\alpha\in (0, 1]$ and $\k\ge 0$. Let $u$ be a  vector field on $\Omega$ and  consider the linear advection-diffusion equation of $\tt$,
\bq\label{ade}
\partial_t\tt+u\cdot\nabla \tt+\k\L^{2\alpha}\tt=0.
\eq
Define 
\bq
B(\Omega)=
\begin{cases}
\left\{v\in L^2(\Omega): \nabla v\in L^\infty(\Omega), \Delta v\in L^q(\Omega), q>2\right\}  \quad\text{if}~d=2,\\
\left\{v\in L^2(\Omega): \nabla v\in L^\infty(\Omega), \Delta v\in L^2(\Omega)\right\} \quad\text{if}~d\ge 3
\end{cases}
\eq
endowed with its natural norm. We prove (see also \cite{ConIgn})
\begin{theo}\label{wp:ade} 
Assume that $u$ is divergence-free and parallel to the boundary, {\it i.e.} $\gamma(u)=0$. 

1. (Existence) Assume $u\in L^1([0, T]; B(\Omega)^d)$ with $T>0$. Equation \eqref{ade} with initial data $\tt_0\in D(\L^2)$ has a solution $\tt$ satisfying
\[
\Vert \tt\Vert_{L^\infty([0, T]; D(\L^2))}+\k \Vert\tt\Vert_{L^2([0, T]; D(\L^{2+\alpha}))}\le C\Vert \tt_0\Vert_{2, D}\exp\big(C\Vert u\Vert_{L^1([0, T]; B(\Omega))}\big).
\]
2. (Uniqueness) Assume $u\in L^2([0, T]; L^\infty(\Omega)^d)$. Equation \eqref{ade} has at most one  weak solution $\tt\in L^\infty([0, T]; L^2(\Omega))$ satisfying
\[
\tt\in L^2([0, T]; H^1_0(\Omega)\cap L^\infty(\Omega)).
\]
\end{theo}
\begin{proof}
1. We proceed as in section 3.1 using the Galerkin approximations. It suffices to derive a priori bounds for $\tt_m\in P_m L^2$ solution to 
\bq\label{Gal:ade}
\begin{cases}
\dot \tt_m+\P_m(u\cdot\nabla\tt_m)+\k \Lambda^{2\alpha}\tt_m=0&\quad t>0,\\
\tt_m=P_m\tt_0&\quad t=0.
\end{cases}
\eq
As in Lemma \ref{lemm:trace}, $u\cdot \nabla\tt_m\in H^1_0(\Omega)$, and hence $u\cdot \nabla\tt_m\in D(\L^2)$. Applying in the first equation of \eqref{Gal:ade} $\L^2=-\Delta$, then taking the scalar product with $\L^2 \tt_m$ and taking into account the fact that $P_m$ is self-adjoint and commutes with $\L^2$ on $D(\L^2)$ we obtain
\begin{align*}
\mez\frac{d}{dt}\Vert \tt_m\Vert_{2, D}^2+\k\Vert \tt_m\Vert_{2+\alpha, D}^2&=\int_\Omega -\Delta(u\cdot\tt_m)\Delta \tt_m dx\\
&=\int_\Omega u\cdot\nabla\L^2\tt_m\L^2 \tt_m dx+\int_\Omega [\L^2, u\cdot\nabla]\tt_m\L^2\tt_m dx.
\end{align*}
Since $\L^2\tt_m$ vanishes on the boundary $\partial\Omega$ and $u$ is divergence-free, an integration by parts gives
\[
\int_\Omega u\cdot\nabla\L^2\tt_m\L^2 \tt_m dx=\mez\int_\Omega u\cdot\nabla\la\L^2\tt_m\ra^2 dx=0.
\]
We recall from \eqref{cmtt:formula} that
\[
[\Delta, u\cdot\nabla]\tt_m=\Delta u\cdot \nabla \tt_m+2\nabla u\cdot \nabla\nabla \tt_m,
\]
hence
\[
\Vert [\Delta, u\cdot\nabla]\tt_m\Vert_{L^2}\le C\Vert u\Vert_{B(\Omega)}\Vert\tt_m\Vert_{H^2}\le C\Vert u\Vert_{B(\Omega)}\Vert\tt_m\Vert_{2, D}.
\]
We obtain thus
\[
\Vert \tt_m\Vert_{L^\infty([0, T]; D(\L^2))}+\k\Vert \tt_m\Vert_{L^2([0, T]; D(\L^{2+\alpha}))}\le C\Vert \tt_0\Vert_{2, D}\exp\big(C\Vert u\Vert_{L^1([0, T]; B(\Omega))}\big).
\]
Passing to the limit $m\to\infty$ can be done by means of the Aubin-Lions lemma (\cite{Lions}). 

2. Under the assumed regularity of $u$ and $\tt$, equation \eqref{ade} is satisfied in $L^2([0, T]; H^{-1}(\Omega))$:
\[
\partial_t\tt+\cnx(u\tt)+\k\L^{2\alpha}\tt=0.
\]
 In addition, $\tt\in L^\infty([0, T]; H^1_0(\Omega))\subset L^2([0, T]; H^1_0(\Omega))$, hence $\tt\in C([0, T]; L^2(\Omega))$ and for {\it a.e.} $t\in [0, T]$ (see Chapter 2, \cite{BoFa})
\[
\mez\frac{d}{dt}\Vert \tt\Vert^2_{L^2}=\langle \partial_t\tt, \tt\rangle_{H^{-1}, H^1}=-\langle \cnx(u\tt), \tt\rangle_{H^{-1}, H^1_0}-\k\langle \L^{2\alpha}\tt, \tt\rangle_{H^{-1}, H^1_0}=(u\tt\cdot, \nabla\tt)-\k\Vert\tt\Vert_{D(\L^\alpha)}^2.
\]
Since $\tt\in H^1(\Omega)\cap L^\infty(\Omega)$, $|\tt|^2\in H^1(\Omega)$. The Stokes formula  then yields
\[
 (u\tt\cdot,\nabla\tt)=(u\cdot, \mez \nabla |\tt|^2)=-(\cnx u, |\tt|^2)=0
\]
for $\cnx u=0$ and $\gamma(u)=0$. Consequently,
\[
\frac{d}{dt}\Vert \tt(t)\Vert^2_{L^2}\le 0
\]
and thus $\tt(t)=0$ for $t\in [0, T]$ if $\tt(0)=0$.
\end{proof}

{\bf{Acknowledgment.}} The work of PC was partially supported by  NSF grant DMS-1209394


\begin{thebibliography}{10}
\small

\bibitem{BCD}
Hajer Bahouri, Jean-Yves Chemin, and Rapha{\"e}l Danchin,
\newblock {\em Fourier analysis and nonlinear partial differential equations},
  volume 343 of {\em Grundlehren der Mathematischen Wissenschaften [Fundamental
  Principles of Mathematical Sciences]}.
\newblock Springer, Heidelberg, 2011.

\bibitem{BoFa}
F. Boyer, P. Fabrie,
\newblock{\em Elements of analysis for the study of some incompressible flow models of viscous fluids}, 
\newblock  Springer-Verlag, Berlin, 2006.
  
\bibitem{CabTan}
X. Cabre, J. Tan,
\newblock Positive solutions of nonlinear problems involving the square root of the Laplacian.
\newblock{\em  Adv. Math.} 224 (2010), no. 5, 2052-2093.

\bibitem{CaVa}
L.  Caffarelli, A. Vasseur, Alexis, Drift diffusion equations with fractional diffusion and the quasi-geostrophic equation. {\em Ann. of Math.} 171 (2010), no. 3, 1903--1930. 

\bibitem{ccw}
P. Constantin, D. Cordoba, and J. Wu,
\newblock On the critical dissipative quasi-geostrophic equation. 
\newblock{\em Indiana Univ. Math. J.}50 (Special Issue): 97--107, 2001. Dedicated to Professors Ciprian Foias and Roger Temam (Bloomington, IN, 2000).

\bibitem{ConIgn}
P. Constantin, M. Ignatova,
\newblock Remarks on the fractional Laplacian with Dirichlet boundary conditions and applications.
\newblock{\em Internat. Math. Res. Notices}, (2016), 1--21.

\bibitem{ConIgn2}
P. Constantin, M. Ignatova,
\newblock Critical SQG in bounded domains.
\newblock {\em Annals of PDE}, (2016) 2:8.

\bibitem{cmt} P.~Constantin, A.J. Majda, and E.~Tabak, Formation of strong fronts in the {$2$}-{D} quasigeostrophic thermal active scalar. 
{\em Nonlinearity}, 7(6) (1994), 1495-1533.

\bibitem{ConNgu}
P.~Constantin, H. Q.~Nguyen,
\newblock Global weak solutions for SQG in bounded domains.
\newblock  arXiv:1612.02489 [math.AP], {\em Comm. Pure Appl. Math.}, to appear 2017.

\bibitem{ConVic}
P. Constantin, V. Vicol, Nonlinear maximum principles for dissipative linear nonlocal operators and applications. {\em Geom. Funct. Anal.} 22 (2012), no. 5, 1289--1321. 
\bibitem{ConWu}
P. Constantin, J. Wu,
\newblock Behavior of solutions of 2D quasi-geostrophic equations. 
\newblock{\em SIAM J. Math. Anal.} 30 (1999), 937--948.

\bibitem{CorCor}
A. Cordoba, D. Cordoba,
\newblock A maximum principle applied to quasi-geostrophic equations.
\newblock{\em Commun. Math. Phys.} 249 (2004), 511--528.

\bibitem{held} I.M. Held, R.T. Pierrehumbert, S.T. Garner, and K.L. Swanson, Surface quasi-geostrophic dynamics.  {\em J. Fluid Mech.}, 282 (1995),1--20.

\bibitem{KP} T. Kato, G. Ponce,
Commutator estimates and the Euler and Navier-Stokes equations. {\em Comm. Pure Appl. Math.} 41 (1988), no. 7, 891--907. 

\bibitem{Kis}
A. Kiselev, F. Nazarov, A. Volberg,
Global well-posedness for the critical 2D dissipative quasi-geostrophic equation. 
{\em Invent. Math.} 167 (2007), no. 3, 445--453. 

\bibitem{Kis2}
A. Kiselev, F. Nazarov, 
A variation on a theme of Caffarelli and Vasseur. 
{\em Zap. Nauchn. Sem. S.-Peterburg. Otdel. Mat. Inst. Steklov.}  370 (2009).

\bibitem{Lions}
J.L. Lions, 
\newblock{\em Quelque methodes de r\'esolution des problemes aux limites non lin\'eaires}. 
\newblock Paris: Dunod-Gauth, 1969.

\bibitem{LioMag}
J. L. Lions, E. Magenes, 
\newblock{\em  Non-homogeneous boundary value problems and applications. Vol. I}. Translated from the French by P. Kenneth. 
\newblock{ Die Grundlehren der mathematischen Wissenschaften, Band 181. Springer-Verlag, New York-Heidelberg}, 1972. 

%\bibitem{Miu}
%H. Miura,
%\newblock Dissipative quasi-geostrophic equation for large initial data in the critical Sobolev space.
%\newblock{\em Commun. Math. Phys.} 255 (2005), 161--181.

\bibitem{Ng}
H. Q. Nguyen, {\it Global weak solutions for generalized SQG in bounded domains}.  arXiv:1704.01462, (2017).

\bibitem{Res}
S. Resnick, 
\newblock Dynamical problems in nonlinear advective partial differential equations,  ProQuest LLC, Ann Arbor, MI, 1995, Thesis (Ph.D.)--The University of Chicago.

\bibitem{Shen}
Z. Shen,
\newblock Bounds of Riesz transforms on Lp spaces for second order elliptic operators. 
\newblock{\em Ann. Inst. Fourier (Grenoble)} 55 (2005), no. 1, 173--197. 

\bibitem{Ju1}
 N. Ju,
 \newblock Existence and uniqueness of the solution to the dissipative 2D quasi-geostrophic equations in the Sobolev space.
 \newblock{\em  Comm. Math. Phys.} 251 (2004), no. 2, 365--376


\end{thebibliography}
\end{document}